\documentclass[12pt,a4paper]{article}
\usepackage[utf8]{inputenc}
\usepackage{geometry}
\pdfoutput=1
\usepackage[english]{babel}
\usepackage[dvipsnames]{xcolor}
\usepackage{amssymb}
\usepackage{amsmath}
\usepackage{amsfonts}
\usepackage{algorithm}
\usepackage{algorithmicx}
\usepackage{enumitem}
\usepackage{algpseudocode}
\usepackage{float}
\usepackage{graphics}
\usepackage{graphicx}
\usepackage{url}
\usepackage{pdfpages}

\usepackage{chngcntr}
\usepackage{ulem}
\usepackage{color}
\newcommand{\blue}[1]{\textcolor{black}{#1}}
\usepackage{hyperref}
\hypersetup{
	colorlinks=true,       
	linkcolor=black,          
	citecolor=black,        
	filecolor=black,      
	urlcolor=black          
}
\usepackage{makeidx}
\usepackage{mathrsfs} 
\usepackage{tabularx}

\usepackage{latexsym} 
\usepackage{theorem} 

\usepackage{cellspace}
\usepackage{tikz}
\usetikzlibrary{positioning} 
\usetikzlibrary{shapes}
\tikzstyle{Template_A}=[draw,circle,thick, minimum size = 0.5cm, color=black,inner sep=0pt]
\tikzstyle{Template_B}=[color=black,draw]
\tikzstyle{vertex}=[draw,circle,fill=black, color=black,inner sep=1.5pt]
\tikzstyle{vertex2}=[draw,circle,fill,inner sep=2pt]
\tikzstyle{vertex3}=[draw,circle,fill,inner sep=0.5pt]
\tikzstyle{S&L}=[draw,circle,minimum height=0.5cm]
\tikzstyle{labeled}=[draw,circle,color=black,inner sep=0.25pt]
\tikzstyle{labeledB}=[draw,circle,inner sep=0.25pt]

\newtheorem{DE}{Definition}[section]

\newtheorem{theorem}[DE]{Theorem}
\newtheorem{lemma}[DE]{Lemma}
\newtheorem{conjecture}[DE]{Conjecture}
\newtheorem{corollary}[DE]{Corollary}

{\theoremstyle{break}\theorembodyfont{\rmfamily}}
{\theoremstyle{break}\theorembodyfont{\rmfamily}}

\newcounter{claim}
\counterwithin{claim}{section}
\newenvironment{proof}[1][]%
{\noindent{\it Proof. }{#1}{}}{\qed\vspace{2ex}}
\newenvironment{claim}[1][]%
{\refstepcounter{claim}\vspace{1ex} {(\it\arabic{section}.\arabic{claim}) {#1}{}}\it}{\vspace{2ex}}
\newenvironment{proofclaim}[1][]%
{\noindent {}{#1}}{This proves~(\arabic{section}.\arabic{claim}).\vspace{2ex}}

\newcommand {\sm} {\setminus}

\newcommand{\qed}{\relax\ifmmode\hskip2em\Box\else\unskip\nobreak\hfill$\Box$\fi}

\author{Ch\'inh T. Hoàng\thanks{Department of Physics and Computer Science, Wilfrid Laurier University, 75 University Avenue West Waterloo, Ontario, Canada, N2L3C5}~ and Cléophée Robin\footnotemark[1]}
\title{A Closure Lemma for tough graphs and Hamiltonian degree conditions\\\small}

\begin{document}
	
	\maketitle	
	\begin{abstract}
			The closure of a graph $G$ is the graph $G^*$ obtained from $G$ by repeatedly adding edges between pairs of non-adjacent vertices whose degree sum is at least $n$, where $n$ is the number of vertices of $G$. The well-known Closure Lemma proved by Bondy and Chv\'atal states that a graph $G$ is Hamiltonian if and only if its closure $G^*$ is. This lemma can be used to prove several classical results in Hamiltonian graph theory. We prove a version of the Closure Lemma for tough graphs. A graph $G$ is $t$-tough if for any set $S$ of vertices of $G$, the number of components of $G-S$ is at most $\frac{|S|}{t}$. A Hamiltonian graph must necessarily be 1-tough. Conversely, Chv\'atal conjectured that there exists a constant $t$ such that every $t$-tough graph is Hamiltonian. The {\it $t$-closure} of a graph $G$ is the graph $G^{t*}$ obtained from $G$ by  repeatedly adding edges between pairs of non-adjacent vertices whose degree sum is at least $n-t$. We prove that, for $t\geq 2$, a $\frac{3t-1}{2}$-tough graph $G$ is Hamiltonian if and only if its $t$-closure $G^{t*}$ is. Ho\`ang conjectured the following: Let $G$ be a graph with degree sequence $d_1 \leq d_2 \leq \ldots \leq d_n$; then $G$ is Hamiltonian if $G$ is $t$-tough and, for all $ i <\frac{n}{2},\mbox{ if } d_i\leq i \mbox{ then } d_{n-i+t}\geq n-i$. This conjecture is analogous to the well known theorem of Chv\'atal on Hamiltonian ideals. Ho\`ang proved the conjecture for $t \leq 3$. Using the closure lemma for tough graphs, we prove  the conjecture for $t = 4$.
	\end{abstract}
	
	\section{Introduction}
	A graph $G$ is {\it Hamiltonian} if it contains an {\it Hamiltonian cycle}, that is a cycle (not necessarily induced) containing all vertices of $G$. A {\it Hamiltonian path} in $G$ is a path (not necessarily induced) containing all vertices of $G$.

	Let $G=(V,E)$ be a graph and $X$ be subset of $V$. We denote by $G[X]$, the subgraph of $G$ induced by $X$. A subset $C$ of vertices of $G$ is called a {\it cut-set} of $G$ if $G[V\sm C]$ is disconnected. A graph $G$ is said to be $t$-tough if for any cut-set $C$ of $G$, $ tc(G[V\sm C])\leq |C|$ where $c(H)$ denotes the number of components of a graph $H$. In this paper, we consider only graphs on at least three vertices. 
	
	In 1973, Chv\'atal conjecture the following:
	
	\begin{conjecture}[Chv\'atall~\cite{chvatal_tough_1973}]\label{c:chvatal}
		
		There exists a constant $t$ such that every $t$-tough graph is Hamiltonian.
	\end{conjecture}
	
	This conjecture is still open. Bauer, Broersma and Veldman~\cite{bauer_2000} showed there are ($\frac{9}{4} - \epsilon$)-tough graphs that are not Hamiltonian for any $\epsilon > 0$. Conjecture~\ref{c:chvatal} was shown to be true for chordal graphs (\cite{chen_1998}, \cite{kabela_2017}). A survey on Chv\'atal's conjecture was given in \cite{bauer_2000}.

	We will investigate a more restricted version of Chv\'atal's conjecture. Our result is related to another well known theorem of Chv\'atal on degree sequences of Hamiltonian graphs.

	A non-decreasing sequence $d_1,d_2,\dots, d_n$ of nonnegative integers is called the degree sequence of a graph $G$ if the vertices of $G$ can be arranger into a sequence $v_1,v_2,\dots, v_n$ such that the degree of each $v_i$ is $d_i$. In 1972, Chv\'atal proved the following theorem. 
	
	\begin{theorem}[Chv\'atal \cite{chvatal_hamiltons_1972}]\label{t:chvatal}
		Let $G$ be a graph with degree sequence $d_1, d_2,\dots , d_n$. If  $d_i\leq i<\frac{n}{2}$ implies $d_{n-i}\geq n - i$ for all $i$, then $G$ is Hamiltonian.
	\end{theorem}
	
	 We say that a graph $G'$ with degree sequence $d'_{1}, d'_{2}, \ldots , d'_{n}$ {\it degree-majorizes} a graph $G$ with degree sequence  $d_1, d_2,\dots , d_n$ if $d'_i \geq d_{i}$ for all $i$. Theorem~\ref{t:chvatal} is best possible in the following sense. If a graph $G$ does not satisfy the hypothesis of Theorem \ref{t:chvatal}, then there is a graph $G'$ that degree-majorizes $G$ such that $G'$ is not Hamiltonian.
	
	Consider a graph $G$ with degree sequence $d_1,d_2,\dots d_n$, and define the predicate $P(t)$ as follows: 
	
	$$P(t): \, \mbox{ for all } i <\frac{n}{2},\mbox{ if } d_i\leq i \mbox{ then } d_{n-i+t}\geq n-i$$.
	
	In this paper, we investigate the following conjecture:

	\begin{conjecture}[Hoàng~\cite{hoang_Hamiltonian_1995}]\label{t:conj_generale}
		If $G$ is a $t$-tough graph satisfying $P(t)$ then $G$ is Hamiltonian. 
	\end{conjecture}
	
	In 1993, Ho\`ang~\cite{hoang_Hamiltonian_1995} proved the following theorem using the well known Closure Lemma due to Bondy and Chv\'atal~\cite{bondy_method_1976}.
	
	\begin{theorem}[Hoàng~\cite{hoang_Hamiltonian_1995}]\label{t:Chinh}
		For $t\leq 3$,	if $G$ is a $t$-tough graph satisfying $P(t)$ then $G$ is Hamiltonian. 
	\end{theorem}
	
	If a graph satisfies $P(0)$ then it satisfies $P(1)$. Since a Hamiltonian graph must be 1-tough, Theorem \ref{t:Chinh} describes a class of Hamiltonian graphs that contains the class of graphs described by Theorem \ref{t:chvatal}. That is, Theorem \ref{t:Chinh} generalizes Theorem \ref{t:chvatal}.
	
	The {\it closure} of an $n$-vertex graph $G$ is the graph $G^*$ obtained from $G$ by the following operations: (i) first let $G^* = G$; (ii) find two non-adjacent vertices $x$, $y$ of $G^*$ such that $d_{G^*}(x) + d_{G^*}(y) \geq  n$, and add the edge $xy$ to $G^*$; (iii) repeat steps (ii) until no such vertices $x$,$y$ are present in $G^*$.

	\begin{lemma}[The Closure Lemma \cite{bondy_method_1976}]\label{l:closure-original}
		A graph $G$ is Hamiltonian if and only if its closure $G^*$ is.
	\end{lemma}
	
	The Closure Lemma can be used to prove several classical results in Hamitonian graph theory. It proved to be fruitful in proving many deep results on cycles, paths and other subgraphs. A survey on the closure concept was given in \cite{broersma_2000}. 
	
	In this paper, we prove the following : 
	
	\begin{theorem}\label{t:notre}  
		If $G$ is a $4$-tough graph satisfying $P(4)$, then $G$ is Hamiltonian. 
	\end{theorem}
	
	To do so, we prove a version of the Closure Lemma for tough graphs. 
	
	The {\it $t$-closure} of a $n$-vertex graph $G$ is the graph $G^{*t}$ obtained from $G$ by the following operations: (i) first let $G^{*t} = G$; (ii) find two non-adjacent vertices $x$, $y$ of $G^*$ such that $d_{G^{*t}}(x) + d_{G^{*t}}(y) \geq  n-t$, and add the edge $xy$ to $G^{*t}$; (iii) repeat steps (ii) until no such vertices $x$,$y$ are present in $G^{*t}$. We will establish the following result.

	\begin{lemma}[$t$-Closure Lemma]\label{l:t-closure}
		
		For $t\geq 2$, a $\frac{3t-1}{2}$-tough graph $G$ is Hamiltonian if and only if its $t$-closure $G^{*t}$ is.
	\end{lemma}

	We did not investigate the tightness of the bound  $\frac{3t-1}{2}$. We think that it is the best possible for small $t$. But finding such counter examples seems to be very difficult.
	
Before establishing the results, we need to introduce a few notions.
	
	When $P$ is a path with vertices $x = v_1, \dots, v_k = y$ and edges  $v_i v_{i+1}$ for $i = 1, 2, \dots, k-1$, for $i < j$ we denote by $P(v_i,v_j)$, the subpath $v_i, v_{i+1}, \dots , v_j$. We will use the same definition when $P$ is a cycle. Observe that, in the context of Hamiltonian paths and cycles, all vertices of the graphs are included in $P$ but some edges may be not included in $P$. Hence, for sake of simplicity, we denote by $v_iv_j\in P$ when $v_iv_j$ is an edge included in $P$. 
	
	When $A$ and $B$ are sets of vertices, we say that $A$ is {\it complete} to $B$ when for any vertex $a\in A$ and any vertex $b\in B$ with $a\neq b$, $a$ and $b$ are adjacent. Observe that this definition holds event when $A\cap B\neq \emptyset$. For the sake of simplicity, we say that a vertex $x$ is complete to $A$ when $\{x\}$ is complete to $A$. 
	
	For a vertex $x$, $N(x)$ denote the set of neigbours of $x$. When $A$ is a subset of vertices, $N_A(x)=N(x)\cap A$ and $N(A)=\cup_{x\in A}N_{V(G)\sm A}(x)$.  
	
	For terms and notations not defined here, we rely on \cite{bondy_graph_2008}.

	In section~\ref{s:t-closure}, we prove the $t$-closure lemma.
	
	In section~\ref{s:Hamiltonian_and_tough}, we prove Theorem~\ref{t:notre}.

\section{The $t$-Closure Lemma}\label{s:t-closure}

	To prove the $t$-Closure Lemma, we divide the proof into two cases, $t=2$ (Section \ref{ss:2tough}) and $t\geq 3$ (Section \ref{ss:t>2}). In addition, we prove a strongest result when $t=2+\epsilon$ for $\epsilon\geq \frac{1}{4}$ (Section \ref{ss:2+e}). To prove those three cases we need preliminary results given in \ref{ss:prem_res}. 
	
	\subsection{Preliminary results}\label{ss:prem_res}
	Let $G$ be a non-Hamitonian graph (i.e., $G$ is not Hamiltonian). Let $P$ be a Hamiltonian path in $G$ with vertices $x = v_1, \dots, v_n = y$ and edges  $v_i v_{i+1}$ for $i = 1, 2, \dots, n-1$. The proofs of the results below are  illustrated in Figure~\ref{f:closure_arretes_interdites}.

	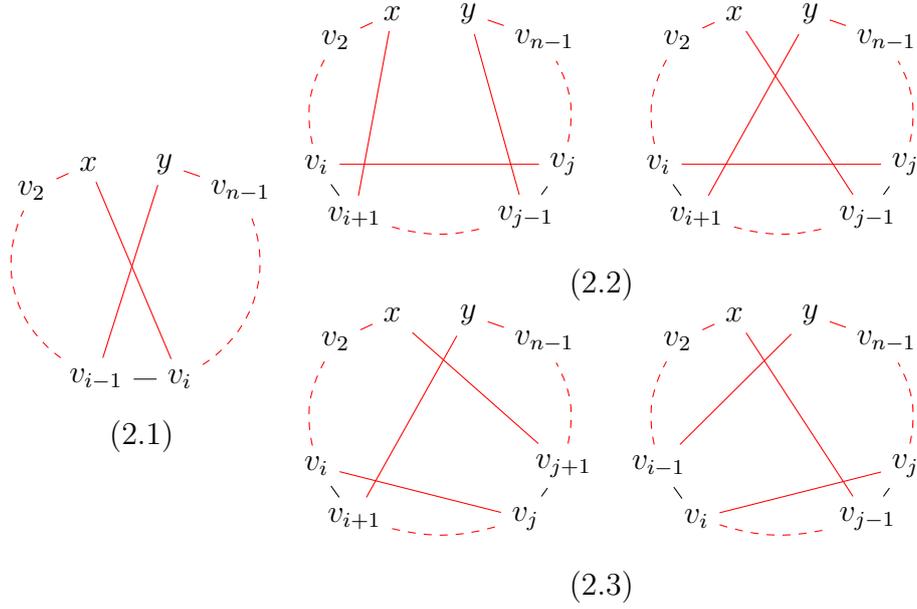
\begin{figure}
		\centering
		\begin{tikzpicture}[scale=1]
			
			\begin{scope}[xshift=0.5cm,yshift=-2cm]
				\node[-] (x) at (0,0.1) {$x$};
				\node[-] (a) at (-0.75,-0.25) {$v_2$};
				\node[-] (b) at (0.1,-2.75) {$v_{i-1}$};
				\node[-] (c) at (1.2,-2.75) {$v_i$};
				\node[-] (d) at (2,-0.25) {$v_{n-1}$};
				\node[-] (y) at (1,0.1) {$y$};
				
				\node[-] (name) at (0.7,-3.5) {(\ref{c:arretes_qui_croise})};
				
				\draw[color=red] (x) -- (a);
				\draw[dashed,color=red] (b) to[bend left=40] (a);
				\draw[-] (b) -- (c);
				\draw[dashed,color=red] (c) to[bend right=40] (d);
				\draw[color=red] (d) -- (y);
				
				\draw[color=red] (x) -- (c);
				\draw[color=red] (y) -- (b);
				
			\end{scope}

			\begin{scope}[xshift=9cm,yshift=0cm]
				\node[-] (x) at (0,0.1) {$x$};
				\node[-] (a) at (-0.75,-0.25) {$v_2$};
				\node[-] (b) at (-1,-1.9) {$v_{i}$};
				\node[-] (c) at (-0.5,-2.6) {$v_{i+1}$};
				\node[-] (d) at (1.75,-2.6) {$v_{j-1}$};
				\node[-] (e) at (2.25,-1.9) {$v_{j}$};
				\node[-] (f) at (2,-0.25) {$v_{n-1}$};
				\node[-] (y) at (1,0.1) {$y$};

				\draw[color=red] (x) -- (a);
				\draw[dashed,color=red] (b) to[bend left=20] (a);
				\draw[-] (b) -- (c);
				\draw[dashed,color=red] (c) to[bend right=15] (d);
				\draw[-] (d)-- (e);
				\draw[dashed,color=red] (e) to[bend right=20] (f);
				\draw[color=red] (y) -- (f);
				
				\draw[color=red] (x) -- (d);
				\draw[color=red] (y) -- (c);
				\draw[color=red] (b) -- (e);
				
			\end{scope}

			\begin{scope}[xshift=4.5cm,yshift=0cm]
				\node[-] (x) at (0,0.1) {$x$};
				\node[-] (a) at (-0.75,-0.25) {$v_2$};
				\node[-] (b) at (-1,-1.9) {$v_{i}$};
				\node[-] (c) at (-0.5,-2.6) {$v_{i+1}$};
				\node[-] (d) at (1.75,-2.6) {$v_{j-1}$};
				\node[-] (e) at (2.25,-1.9) {$v_{j}$};
				\node[-] (f) at (2,-0.25) {$v_{n-1}$};
				\node[-] (y) at (1,0.1) {$y$};
				
				\node[-] (name) at (2.75,-3.5) {(\ref{c:arretes_biais2})};
				
				\draw[color=red] (x) -- (a);
				\draw[dashed,color=red] (b) to[bend left=20] (a);
				\draw[-] (b) -- (c);
				\draw[dashed,color=red] (c) to[bend right=15] (d);
				\draw[-] (d)-- (e);
				\draw[dashed,color=red] (e) to[bend right=20] (f);
				\draw[color=red] (y) -- (f);
				
				\draw[color=red] (x) -- (c);
				\draw[color=red] (y) -- (d);
				\draw[color=red] (b) -- (e);
				
			\end{scope}

			\begin{scope}[xshift=4.5cm,yshift=-4cm]
				\node[-] (x) at (0,0.1) {$x$};
				\node[-] (a) at (-0.75,-0.25) {$v_2$};
				\node[-] (b) at (-1,-1.9) {$v_{i}$};
				\node[-] (c) at (-0.5,-2.6) {$v_{i+1}$};
				\node[-] (d) at (1.75,-2.6) {$v_{j}$};
				\node[-] (e) at (2.25,-1.9) {$v_{j+1}$};
				\node[-] (f) at (2,-0.25) {$v_{n-1}$};
				\node[-] (y) at (1,0.1) {$y$};
				
				\node[-] (name) at (2.75,-3.5) {(\ref{c:triples})};
				
				\draw[color=red] (x) -- (a);
				\draw[dashed,color=red] (b) to[bend left=20] (a);
				\draw[-] (b) -- (c);
				\draw[dashed,color=red] (c) to[bend right=15] (d);
				\draw[-] (d)-- (e);
				\draw[dashed,color=red] (e) to[bend right=20] (f);
				\draw[color=red] (y) -- (f);
				
				\draw[color=red] (x) -- (e);
				\draw[color=red] (y) -- (c);
				\draw[color=red] (b) -- (d);
				
			\end{scope}

			\begin{scope}[xshift=9cm,yshift=-4cm]
				\node[-] (x) at (0,0.1) {$x$};
				\node[-] (a) at (-0.75,-0.25) {$v_2$};
				\node[-] (b) at (-1,-1.9) {$v_{i-1}$};
				\node[-] (c) at (-0.5,-2.6) {$v_i$};
				\node[-] (d) at (1.75,-2.6) {$v_{j-1}$};
				\node[-] (e) at (2.25,-1.9) {$v_{j}$};
				\node[-] (f) at (2,-0.25) {$v_{n-1}$};
				\node[-] (y) at (1,0.1) {$y$};

				\draw[color=red] (x) -- (a);
				\draw[dashed,color=red] (b) to[bend left=20] (a);
				\draw[-] (b) -- (c);
				\draw[dashed,color=red] (c) to[bend right=15] (d);
				\draw[-] (d)-- (e);
				\draw[dashed,color=red] (e) to[bend right=20] (f);
				\draw[color=red] (y) -- (f);
				
				\draw[color=red] (x) -- (d);
				\draw[color=red] (y) -- (b);
				\draw[color=red] (e) -- (c);
				
			\end{scope}
		\end{tikzpicture}
		\caption{Illustrations of (\ref{c:arretes_qui_croise}),(\ref{c:arretes_biais2}) and (\ref{c:triples})}
	\end{figure}\label{f:closure_arretes_interdites}

	\vspace{2ex}
	\begin{claim}\label{c:arretes_qui_croise}
		For all $i$ with $2\leq i \leq n$, if $v_ix\in E$ then $v_{i-1}y\notin E$.  
	\end{claim}
	
	\begin{proofclaim}
		Suppose $v_ix\in E$ and $v_{i-1}y\in E$.  Now, $x, v_2, \ldots, v_{i-1} , y, v_{n-1}, \ldots v_i, x$ is a Hamiltonian cycle in $G$, a contradiction.
	\end{proofclaim} 
	
	\begin{claim}\label{c:arretes_biais2}
		For all edges $v_iv_j\in E$ with $1\leq i<i+1<j \leq n$, 
		
		if $v_{i+1}x\in E$ then $v_{j-1}y\notin E$ and 
		
		if $v_{i+1}y\in E$ then $v_{j-1}x\notin E$.
	\end{claim}
	
	\begin{proofclaim} 
		Suppose that $v_iv_j\in E$.
		
		If $xv_{i+1}\in E$ and $v_{j-1}y\in E$ then $x, v_2, \ldots, v_i$, $v_j, v_{j+1} , \ldots, y, v_{j-1}$, $v_{j-2}, \ldots, v_{i+1} , x $ is a Hamiltonian cycle in $G$, a contradiction. 
		
		If $yv_{i+1}\in E$ and $v_{j-1}x\in E$ then $x, v_2, \dots, v_i, v_j$,$v_{j+1}, \ldots, y$, $v_{i+1}, v_{i+2}, \ldots, v_{j-1}, x$ is a Hamiltonian cycle in $G$, a contradiction.
	\end{proofclaim}
	
	\begin{claim}\label{c:triples} 	
		For all edges $v_iv_j\in E$ with $1\leq i<j \leq n$, 
		
		if $v_{i+1}y\in E$ then $v_{j+1}x\notin E$ and
		
		if $v_{i-1}y\in E$ then $v_{j-1}x\notin E$.
	\end{claim}
	
	\begin{proofclaim}
		We note the two cases of (\ref{c:triples}) are symmetrical to the first case of (\ref{c:arretes_biais2}), but for the sake of completeness, we enumerate the vertices of the Hamiltonian cycles below. 
		
		Suppose that $v_iv_j\in E$ with $i<j$. 
		
		If $v_{i+1}y\in E$ and $v_{j+1}x\in E$, then $x, v_2,$ $\ldots, v_i$, $v_j, v_{j-1},$ $\ldots, v_{i+1}, y, v_{n-1}$, $ \ldots, v_{j+1} , x $ is a Hamiltonian cycle in $G$, a contradiction.

		If $v_{i-1}y\in E$ and $v_{j-1}x\in E$, then $x, v_2,$ $\ldots, v_{i-1}, y,$ $v_{n-1}, \ldots, v_j,$ $ v_i, v_{i+1}, \ldots, v_{j-1}x$ is a Hamiltonian cycle in $G$, a contradiction.
	\end{proofclaim}
	
	\begin{claim}\label{c:Sdanssegment} 	
		For all $a<b$ such that $v_ay\in E$ and $v_bx\in E$, there exists $a<s<b$ such that $v_sx,v_sy\notin E$. 
		
		In addition, if $s$ is unique $v_{s-1}y,v_{s+1}x\in E$. 
	\end{claim}

	\begin{proofclaim}
		Let $\alpha$ be the largest integer such that $a<\alpha<b$, $v_{\alpha +1}x\in E$ and $v_{\alpha}x\notin E$. By  (\ref{c:arretes_qui_croise}), such vertex exists and  $v_\alpha y \not\in E$. 
		
		Let $\beta$ be the smallest integer such that $a<\beta<b$, $v_{\beta -1}y\in E$ and $v_{\beta}y\notin E$. By  (\ref{c:arretes_qui_croise}), such vertex exists and $v_\beta x \not\in E$. Now, if $s=\alpha=\beta$, then $v_{s-1}y,v_{s+1}\in E$. 
	\end{proofclaim}

	\subsection{The $2$-closure Lemma}\label{ss:2tough}

	\begin{lemma}\label{l:t=1}
		Let $G=(V,E)$ be a 2-tough graph and let $x$ and $y$ be two non-adjacent vertices of $G$ such that $d(x)+d(y)\geq n-1$. Then $G'=(V(G),E(G)\cup \{xy\})$ is Hamiltonian if and only if $G$ is Hamiltonian.
	\end{lemma}
	\begin{proof}
		
		Let $G$, $x$, $y$ and $G'$ be as in the statement of the lemma. It is obvious that if $G$ is Hamiltonian then $G'$ is also Hamiltonian. In addition, if $d(x)+d(y)\geq n$ then, by the Closure Lemma (Lemma~\ref{l:closure-original}), the result holds. So, we may assume that $d(x)+d(y)=n-1$. 
		
		Suppose, by contradiction that $G'$ is Hamiltonian but $G$ is not Hamiltonian. Let $H$ be a Hamiltonian cycle in $G'$. Observe that $xy\in E(H)$ and that $xy\notin E(G)$, for otherwise $G$ is Hamiltonian. We denote the vertices of the Hamiltonian path $P$ of $G$ by $v_1,v_2,\dots, v_{n-1},v_n$ such that $v_1=x$, $v_n=y$ and $v_iv_{i+1}\in P$ for all $i\leq n-1$.

		\begin{claim}\label{ct2:D_non_vide}
			$|N(x)\cap N(y)|\geq 2$.
		\end{claim}
		
		\begin{proofclaim}
			Suppose that $|N(x)\cap N(y)|\leq 1$. Since $xy \not\in E$, $|N(x)\cup N(y)| \leq n-2$. We also have $d(x) + d(y) = n-1$ and $ d(x) + d(y) = |N(x)\cup N(y)| + |N(x)\cap N(y)|$. Thus, we have $n -1 \leq n-2 + |N(x) \cap N(y)|$, and therefore $|N(x) \cap N(y)| = 1$. Hence there exists $v_i$ such that $N(x)\cap N(y)= \{v_i\}$.
			
			Now observe that $|N(x)\cup N(y)| - |N(x) \cap N(y)| =n-3$ and $|V\sm \{x,y,v_i\}|=n-3$. Therefore, every vertex in $V\sm \{x,y,v_i\}$ is either adjacent to $x$ or adjacent to $y$. 
			
			By (\ref{c:arretes_qui_croise}), we have $N(x)=\{v_2,\dots,v_i\}$ and $N(y)=\{v_i,\dots,v_{n-1}\}$. By (\ref{c:arretes_biais2}), there is no edge between $\{v_1,\dots,v_{i-1}\}$ and $\{v_{i+1},\dots,v_n\}$. Hence $G[V\sm\{v_i\}]$ is disconnected and $G$ is not 2-tough, a contradiction. 
		\end{proofclaim}

		Define the following subsets of $V\sm \{x,y\}$ : $D=N(x)\cap N(y)$ and $S=V\sm (N(x)\cup N(y)\cup \{x,y\})$. In other words, $D$ contains all vertices adjacent to both $x$ and $y$ and $S$ contains all vertices that are adjacent neither to $x$ nor to $y$. By (\ref{ct2:D_non_vide}), $|D|\geq 2$. 
		
		\vspace{2ex}
		\begin{claim}\label{c:Strucutre_voisinage}
			For all $v_i\in S$,  $v_{i-1}y,v_{i+1}x\in E$ and $S$ is a stable set. 
		\end{claim}

		\begin{proofclaim}

			Let  $v_a$ and $v_b$ be two vertices of $D$ such that $a < b$ and  $P(v_a,v_b) \cap D =\{v_a,v_b\}$. That is, the interior vertices of $P(v_a,v_b)$ do not belong to $D$.  There are $|D|-1$ such different pairs of vertices. We will call such path $P(v_a,v_b)$ a {\it segment}. By (\ref{c:arretes_qui_croise}), we have $b\neq a+1$, $v_{a+1}x\notin E$ and $v_{b-1}y\notin E$. 
			
			By (\ref{c:Sdanssegment}), there is a vertex $v_s\in S$ with $a<s<b$. This vertex is unique because there are $|D|-1$ segments $P(v_a,v_b)$ and they do not intersect except on their ends. Every such segment contains at least one vertex from $S$. Observe that $n-1 = d(x)+d(y)=|V|-|\{x,y\}|-|S|+|D| = n-2 -|S| + |D|$. Hence, we have $|S|=|D|-1$. By the pigeon hole principle, $v_s$ is unique for each segment $P(v_a,v_b)$. Again, by (\ref{c:Sdanssegment}) we have $v_{i-1}y,v_{i+1}x\in E$.

			Consider now two different segments $P(v_a,v_b)$ and, $P(v_{a'},v_{b'})$ with vertices $v_s,v_{s'}\in S$ such that $a<s<b\leq a'<s'<b'$. By the argument above, we have $v_{s+1} x , v_{s' -1} y \in E$.  
			By (\ref{c:arretes_biais2}), we have $v_s v_{s'} \not \in E$. Thus, $S$ is a stable set.
		\end{proofclaim}
		
		\begin{claim}\label{c:voisin_S}
			For all $v_i\in S$, 
			
			$N(v_i)\subseteq \{v_j:j<i \mbox{ and }v_{j+1}\in S\}\cup \{v_j:i<j \mbox{ and }v_{j-1}\in S\}$.
		\end{claim}
		
		\begin{proofclaim}
			Suppose that there exists $v_i\in S$ and $v_j\in N(v_i)$. By (\ref{c:Strucutre_voisinage}), we have  $v_{i-1}y\in E$ and $v_{i+1}x\in E$. 
			
			Suppose that $j<i$. By (\ref{c:arretes_biais2}),  we have  $v_{j+1}x\notin E$ and by (\ref{c:triples}),  we have $v_{j+1}y\notin E$. Hence  we have  $v_{j+1}\in S$. 
			
			Suppose now that $i<j$. By (\ref{c:arretes_biais2}),  we have  $v_{j-1}y\notin E$ and by (\ref{c:triples}),  we have  $v_{j-1}x\notin E$. Hence  we have  $v_{j-1}\in S$.
		\end{proofclaim}
		
		By (\ref{c:voisin_S}),  we have $N(S)\subseteq \{v_{i-1},v_{i+1} : v_i\in S\}$. Hence  we have $|N(S)|\leq 2|S|$ and  by (\ref{c:Strucutre_voisinage}), it follows that  $S$ is a stable set. Since $x\notin S\cup N(S)$, the graph $G[V\sm N(S)]$ has at least $|S|+1$ connected components, a contradiction to the fact that $G$ is 2-tough. Therefore $G$ is Hamiltonian. 
		
	\end{proof}

	\subsection{The $(2+\epsilon)$-Closure Lemma}\label{ss:2+e}

	\begin{lemma}\label{l:2+espy}
		Let $G$ be a $(2+\epsilon)$-tough graph ($\epsilon > \frac{1}{4}$)  with $x,y\in V(G)$ such that $d(x)+d(y)\geq n-2$. 
		
		$G$ is Hamiltonian if and only if $G'=G + xy$ is Hamiltonian.
	\end{lemma}

	\begin{proof}
		
		Let $G=(V,E)$, $x$, $y$ and $G'$ be as in the statement of the Lemma. It is obvious that if $G$ is Hamiltonian then $G'$ is also Hamiltonian. In addition, if $d(x)+d(y)\geq n-1$ then, by Lemma~\ref{l:t=1}, the result holds. So, we may assume that $d(x)+d(y)=n-2$. 
		
		Suppose, by contradiction that $G'$ is Hamiltonian but $G$ is not Hamiltonian. Let $H$ be a Hamiltonian cycle in $G'$. Observe that $xy\in H$ and that $xy\notin E$ for otherwise, $G$ is Hamiltonian. We denote the vertices of the Hamiltonian path $P$ of $G$ by $v_1,v_2,\dots, v_{n-1},v_n$ such that $v_1=x$, $v_n=y$ and $v_iv_{i+1}\in E$ for all $i\leq n-1$.
		
		Define the following subsets of $V\sm \{x,y\}$ : $D=N(x)\cap N(y)$ and $S=V\sm (N(x)\cup N(y)\cup \{x,y\})$. In other words, $D$ contains all vertices adjacent to both $x$ and $y$ and $S$ contains all vertices that are adjacent neither to $x$ nor to $y$.  Let $S^*$ be a subset of $S$ such that $S^*=\{v_i\in S : v_{i-1}y,v_{i+1}x\in E \}$.
		
		A path $P(v_a,v_b)$ is a segment if  $P(v_a,v_b) \cap D =\{v_a,v_b\}$. There are $|D|-1$ segments.

		\begin{claim}\label{c2:voisin_S*}
			For all $v_i\in S^*$, 
			
			$N(v_i)\subseteq \{v_j:j<i \mbox{ and }v_{j+1}\in S\}\cup \{v_j:i<j \mbox{ and }v_{j-1}\in S\}$.
		\end{claim}
		
		\begin{proofclaim}
			Suppose that there exists $v_i\in S^*$ and $v_j\in N(v_i)$. By definition, we have $v_{i-1}y\in E$ and $v_{i+1}x\in E$. 
			
			 By (\ref{c:arretes_biais2}), if $j<i$ then $v_{j+1}x\notin E$ and by (\ref{c:triples}), $v_{j+1}y\notin E$. Hence if $j<i$, then $v_{j+1}\in S$. 
			
			 By (\ref{c:arretes_biais2}), if $i<j$ then $v_{j-1}y\notin E$ and by (\ref{c:triples}), $v_{j-1}x\notin E$. Hence if $i<j$, then $v_{j-1}\in S$.
		\end{proofclaim}

		\begin{claim}\label{c:borne_sur_D}
			$|S|=|D|$ and $|D|\geq 2$. 
		\end{claim}
		
		\begin{proofclaim}
			Observe that $n-2 = d(x)+d(y)=|V\sm\{x,y\}|-|S|+|D| = n-2 -|S|+|D| $ and so $|S|=|D|$. 
			
			Suppose first that $|D|=0$ and so  $N(x)\cup N(y)=V \sm \{x,y\}$. By (\ref{c:arretes_qui_croise}), there exists $v_k$ such that $N(x)=\{v_2,\dots v_k\}$ and $N(y)=\{v_{k+1},\dots v_{n-1}\}$. By (\ref{c:arretes_biais2}), there is no edge between $\{v_2,\dots v_{k-1}\}$ and $\{v_{k+2},\dots v_{n-1}\}$. Therefore $G[V\sm \{v_k,v_{k+1}\}]$ has two connected components, a contradiction to the fact that  $G$ is $(2+\epsilon)$-tough.
			
			Suppose now that $|D|=1=|S|$ and so $|N(x)\cup N(y)|=n-3$.

			Let $v_d\in D$ and $v_s\in S$. The two cases being symmetrical, we may suppose that $d<s$. 
			
			
			By repeated applications of (\ref{c:arretes_qui_croise}), we have $\{v_1, \ldots, v_d\}\subseteq N(x)$ and 
			$\{v_d,\dots,v_{s-1}\}\subseteq N(y)$. By (\ref{c:arretes_qui_croise}), $v_s v_j \not\in E$ for $j = 2, \dots, d-1$.

		Suppose  $v_s\in S^*$. By (\ref{c:arretes_biais2}),  we have  $v_s v_j \not\in E$ for $j= d, d+1, \ldots, s-2$. Suppose $v_s v_j \in E$ for some $j \in \{ s+2, \ldots, n-1 \}$. Considering the edge $v_{s-1}y$, by (\ref{c:triples}),  we have  $v_{j-1} x \not\in E$. This implies $v_{j-1} y \in E$, a contradiction to (\ref{c:arretes_biais2}). Hence $N(v_s)\subseteq \{v_{s-1},v_{s+1}\}$, a contradiction to $G$ being $(2+\epsilon)$-tough. Therefore $v_s\notin S^*$. Since $v_{s-1}y\in E$, $v_{s+1}y\in E$ and by (\ref{c:arretes_qui_croise}),  we have  $N(x)=\{v_1,\dots v_d\}$ and $N(y)=\{v_{d},\dots v_{n-1}\}\sm \{v_s\}$.

			By (\ref{c:arretes_biais2}), there is no edge between $\{v_2,\dots , v_{d-1}\}$ and $\{v_{d+1},\dots , v_{n-1}\}\sm \{v_{s+1}\}$. Therefore $G[V\sm \{v_{s+1},v_d\}]$ has two connected components, a contradiction to $G$ being $(2+\epsilon)$-tough.
		\end{proofclaim}

		\begin{claim}\label{c:tailleS*}
			$|S^*|=|S|-2$. 
		\end{claim}
		
		\begin{proofclaim}
			By (\ref{c:borne_sur_D}),  we have  $|S|=|D|$. There are $|D|-1$ segments. By (\ref{c:Sdanssegment}), each segment must contain at least one vertex $v_s$ of $S$, and  if $v_s$ is the only vertex in $S$ of this segment, then $v_s$ is in $S^*$. By the pigeon hole principal, we have $|S^*|\geq |S|-2$.
			
			By (\ref{c2:voisin_S*}), we have   $N(S^*)\subseteq \{v_i : v_{i-1}\in S \mbox{ or } v_{i+1}\in S\}$ and $S^*$ is a stable set. By definition,  we know that  $x,y\notin N(S^*)$ and so, $c(G[V\sm N(S^*)])\geq |S^*|+1$. Since $G$ is $(2+\epsilon)$-tough,  we have  $(2+\epsilon)(|S^*|+1)\leq |N(S^*)|\leq 2|S|$. If $|S^*|\geq |S|-1$, then  $(2+\epsilon)|S|\leq 2|S|$, a contradiction to $\epsilon>1/4$.
		\end{proofclaim}
		
		\begin{claim}\label{c2:lesDeux}
			There exists $\alpha$ such that $\{v_\alpha,v_{\alpha+1}\}=S\sm S^*$ and $v_{\alpha -1}y, v_{\alpha +2}x\in E$. 
		\end{claim}
		
		\begin{proofclaim}
			By (\ref{c:tailleS*}), we have  $|S^*|=|S|-2=|D|-2$. There are $|D|-1$ segments. By (\ref{c:Sdanssegment}) and the pigeon hole principal, there exists a unique segment $P(v_a,v_b)$ with two vertices $v_\alpha, v_\beta \in S\sm S^*$ such that $\{v_\alpha, v_\beta \} = \{v_a,\dots,v_b\}\cap S$. Without loss of generality, we may assume $\alpha < \beta$. 
			
			Recall that every vertex in $\{v_a, v_{a+1}, \ldots, v_b\} \sm \{v_{\alpha}, v_\beta\}$ must be adjacent to $x$ or $y$. By repeatedly applying (\ref{c:arretes_qui_croise}), it follows that $y$ is complete to  $\{ v_{a+1}, \ldots, v_{\alpha -1}\}$,  and $x$ is complete to  $\{v_{b-1}, \ldots, v_{\beta+1}\}$. Since $v_\alpha \notin S^*$,  we have  $v_{\alpha +1}x\notin E$. Suppose that $v_{\alpha +1}y \in E$.  By (\ref{c:arretes_qui_croise}), $y$ is adjacent to all vertices in $\{v_{\alpha + 2}, \ldots, v_{\beta -1}\}$.   But this implies $v_\beta\in S^*$, a contradiction. Hence  $v_{\alpha +1}y\notin E$, and therefore $v_{\alpha +1}\in S$ and so $v_{\alpha +1}=v_\beta$.
		\end{proofclaim}

		In the following, we let $\{v_\alpha,v_{\alpha+1}\}=S\sm S^*$.
		
		\vspace{2ex}
		\begin{claim}	\label{c:voisins_Alpha}
			$N_{S^*}(v_\alpha)=\emptyset$ or $N_{S^*}(v_{\alpha+1})=\emptyset$.
		\end{claim}
		\begin{proofclaim}
			
			By (\ref{c2:lesDeux}), we have  $v_{\alpha-1}y, v_{\alpha+2}x\in E$.
			
			By (\ref{c:arretes_qui_croise}) and (\ref{c:triples}),  we have  $N(v_\alpha)\subseteq \{v_i : i<\alpha, v_{i+1}x\notin E\}\cup \{v_i : i\geq \alpha+1, v_{i-1}x\notin E\}$ and $N(v_{\alpha+1})\subseteq \{v_i : i\leq\alpha, v_{i+1}y\notin E\}\cup \{v_i, i>\alpha+1, v_{i-1}y\notin E\}$. Since, for all $v_i\in S^*$, $v_{i-1}y\in E$ and $v_{i+1}x\in E$, we have  $N_{S^*}(v_\alpha)\subseteq \{v_i: v_i\in S^*\mbox{ and } i>\alpha+1\}$ and $N_{S^*}(v_{\alpha+1})=\{v_i : v_i\in S^*\mbox{ and } i<\alpha \}$.
			
			Suppose that there exists $v_i\in N_{S^*}(v_\alpha)$ and  $v_j\in N_{S^*}(v_{\alpha+1})$. By previous observation, we have  $j<\alpha<\alpha+1<i$. Since $v_i,v_j\in S^*$, we have  $v_{i-1}y\in E$, and $v_{j+1}x\in E$. Now, $x v_{2} \ldots v_j v_{\alpha+1} v_{\alpha+2} \ldots v_{i-1}y v_{n-1} \ldots v_i v_\alpha v_{\alpha - 1} \ldots v_{j+1}x$ is a Hamiltonian cycle in $G$, a contradiction. 
		\end{proofclaim}

		\begin{claim}	
			$|S^*|\geq 1$. 
		\end{claim}
		
		\begin{proofclaim}
			Suppose that $|S^*|=0$. By (\ref{c:tailleS*}),  we have  $|S|=2$, and by (\ref{c:borne_sur_D}),  $|D|= 2$. Let $D =\{v_a,v_b\}$ ($a<b$). Recall that $S=\{v_{\alpha},v_{\alpha +1}\}$. By (\ref{c:arretes_qui_croise}),  we have  $N(x)=\{v_2,\dots,v_a\}\cup \{v_{\alpha+2},\dots v_b\}$ and $N(y)=\{v_a,\dots,v_{\alpha-1}\}\cup \{v_b,\dots v_{n-1}\}$. 
			
			We will prove that there are no edges between $A=N(x)\sm \{v_a,v_b,v_{\alpha+2}\}\cup \{x\}$ and
			\blue{ $B=(N(y)\sm \{v_a,v_b\})\cup \{v_{\alpha },y \}$. Suppose that $v_i\in A$, $v_j\in B$ such that $v_iv_j\in E$. 
			 Suppose that $i<j$. Then we have $v_{i+1} x, v_{j-1} y \in E$, a contradiction to (\ref{c:arretes_biais2}). Suppose that $i>j$. This implies $i \in \{ \alpha + 3, \dots, b-1\}$ and $j \in \{ a+1, \ldots, \alpha \}$. Since $v_{i-1} x, v_{j-1} y \in E$, we have a contradiction to (\ref{c:triples}). }   
			
			
			Hence there is no edge between $A$ and $B$ and so $G[V\sm \{v_a,v_{\alpha+1},v_{\alpha+2},v_b\}]$ has at least two connected components, a contradiction to $G$ being $(2+\epsilon)$-tough.
		\end{proofclaim}
		
		\begin{claim}	
			$|S|=4$
		\end{claim}
		
		\begin{proofclaim}
			
			By (\ref{c2:voisin_S*}) and (\ref{c:voisins_Alpha}), we have $|N(S^*)|\leq 2|S|-1$.
			 Recall that $S^*$ is a stable set by (\ref{c2:voisin_S*}). In addition, we have $x\notin S^*\cup N(S^*)$. Hence $G[V\sm N(S^*)]$ has at least $|S^*|+1$ connected components. Since $G$ is $(2+\epsilon)$-tough, we have $(2+\epsilon)(|S^*|+1)\leq 2|S|-1$. By (\ref{c:tailleS*}), it follows that $(2+\epsilon)(|S|-1)\leq 2|S|-1$ and so $|S|\leq 1+1/\epsilon$. But $\epsilon>1/4$, so  $|S|\leq 4$.
			
			Now, by (\ref{c2:voisin_S*}), for all $v_i\in S^*$, we have $|N(v_i)|\leq |S|+1$. Since $G$ is $(2+\epsilon)$-tough, we have $|N(v_i)|\geq 4+2\epsilon $. Hence $|S|+1\geq 5$ and by a previous observation, we have $|S|=4$.
		\end{proofclaim}
		
		Now we know that $|S|=|D|=4$ and $|S^*|=2$. Let $S=\{v_i,v_j,$ $v_\alpha,v_{\alpha+1}\}$ with $i<j$ and $v_i,v_j\in S^*$.
		
		Since $G$ is $(2+\epsilon)$-tough, every vertex in $G$ has degree at least 5. By (\ref{c2:voisin_S*}), we have  $N(v_i)=\{v_{i-1},v_{i+1},v_{j+1},v_{\alpha+1},v_{\alpha+2}\}$ if $\alpha>i$ or $N(v_i)=\{v_{i-1},v_{i+1},$ $v_{j+1},$ $v_{\alpha-1},v_{\alpha}\}$ if $\alpha<i$, and  $N(v_j)=\{v_{j-1},v_{j+1},v_{i-1},$ $v_{\alpha+1},v_{\alpha+2}\}$ if $\alpha>j$ or $N(v_i)=\{v_{j-1},v_{j+1},v_{i-1},$ $v_{\alpha-1},v_{\alpha}\}$ if $\alpha<j$.
		\blue{If $i < \alpha < j$, then $v_{\alpha +1} v_i , v_{\alpha} v_j \in E$, a contradiction to  (\ref{c:voisins_Alpha})}. Hence either $i<j<\alpha$ or $\alpha<i<j$. 
		
		If $i<j<\alpha$, then $x v_2 \ldots v_i v_{\alpha+1} v_{\alpha} v_{\alpha -1} \ldots v_j v_{\alpha+2} v_{\alpha + 3} \ldots $ $ y v_{j-1} v_{j-2} \ldots v_{i+1}x $ is a Hamiltonian cycle in $G$, a contradiction. So $\alpha<i<j$ and $x v_2 \ldots  v_{\alpha-1} v_i v_{i-1} \ldots v_\alpha v_j  v_{j+1} \ldots yv_{j-1} v_{j-2} \ldots  v_{i+1}x$ is a Hamiltonian cycle in $G$, a contradiction. 
		
	\end{proof}

	\subsection{A closure lemma for t-tough graphs}\label{ss:t>2}
	
	\begin{lemma}\label{l:general}
		Let $G$ be a $\frac{3t-1}{2}$-tough graph ($t\geq 2$, $t$ is an integer) with $x,y\in V(G)$ such that $d(x)+d(y)= n-t$.
		
		$G$ is Hamiltonian if and only if $G^*=(V(G),E(G)\cup \{xy\})$ is Hamiltonian.
	\end{lemma}

	\begin{proof}
		
		Let $G$, $x$, $y$ and $G'$ be as in the statement of the Lemma.
		
		Observe that if $t=2$ then the graph is $\frac{5}{2}$-tough and so, $G$ is also $2+\epsilon$-tough for $\epsilon =1/2>1/4$. By Lemma~\ref{l:2+espy}, the results hold. Therefore, we can assume that $t\geq 3$. It is obvious that if $G$ is Hamiltonian then $G'$ is also Hamiltonian.

		Suppose, by contradiction that $G'$ is Hamiltonian but $G$ is not Hamiltonian. Let $H$ be a Hamiltonian cycle in $G'$. Observe that $xy\in H$ and that $xy\notin E$ for otherwise, $G$ is Hamiltonian. We denote the vertices of the Hamiltonian path $P$ of $G$ by $v_1,v_2,\dots, v_{n-1},v_n$ such that $v_1=x$, $v_n=y$ and $v_iv_{i+1}\in E$ for all $i\leq n-1$.
		
		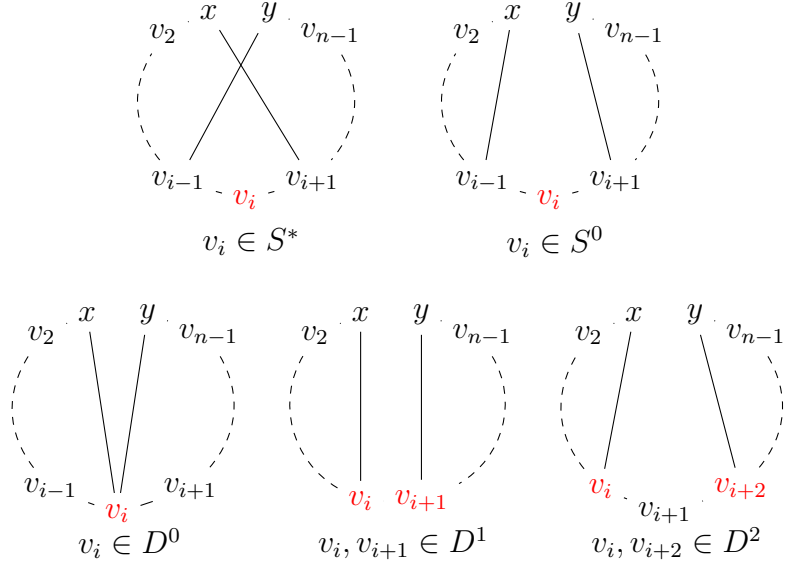
\begin{figure}
			\centering
			\begin{tikzpicture}[scale=0.8]
				
				\begin{scope}[xshift=0cm,yshift=0cm]
					\node[-] (x) at (0,0.1) {$x$};
					\node[-] (a) at (-0.75,-0.25) {$v_2$};
					\node[-] (b) at (-0.5,-2.7) {$v_{i-1}$};
					\node[color=red] (c) at (0.6,-3) {$v_i$};
					\node[-] (d) at (1.7,-2.7) {$v_{i+1}$};
					\node[-] (e) at (2,-0.25) {$v_{n-1}$};
					\node[-] (y) at (1,0.1) {$y$};
					
					\node[-] (name) at (0.7,-3.7) {$v_i\in S^*$};
					
					\draw[-] (x) -- (a);
					\draw[dashed] (b) to[bend left=35] (a);
					\draw[-] (b) -- (c);
					\draw[-] (c) -- (d);
					\draw[dashed] (e) to[bend left=35] (d);
					\draw[-] (e) -- (y);
					
					\draw[-] (x) -- (d);
					\draw[-] (y) -- (b);
					
				\end{scope}

				\begin{scope}[xshift=5cm,yshift=0cm]
					\node[-] (x) at (0,0.1) {$x$};
					\node[-] (a) at (-0.75,-0.25) {$v_2$};
					\node[-] (b) at (-0.5,-2.7) {$v_{i-1}$};
					\node[color=red] (c) at (0.6,-3) {$v_i$};
					\node[-] (d) at (1.7,-2.7) {$v_{i+1}$};
					\node[-] (e) at (2,-0.25) {$v_{n-1}$};
					\node[-] (y) at (1,0.1) {$y$};
					
					\node[-] (name) at (0.7,-3.7) {$v_i\in S^0$};
					
					\draw[-] (x) -- (a);
					\draw[dashed] (b) to[bend left=35] (a);
					\draw[-] (b) -- (c);
					\draw[-] (c) -- (d);
					\draw[dashed] (e) to[bend left=35] (d);
					\draw[-] (e) -- (y);
					
					\draw[-] (x) -- (b);
					\draw[-] (y) -- (d);
					
				\end{scope}

				\begin{scope}[xshift=-2cm,yshift=-5cm]
					\node[-] (x) at (0,0.1) {$x$};
					\node[-] (a) at (-0.75,-0.25) {$v_2$};
					\node[-] (b) at (-0.6,-2.8) {$v_{i-1}$};
					\node[color=red] (c) at (0.5,-3.2) {$v_i$};
					\node[-] (d) at (1.7,-2.8) {$v_{i+1}$};
					\node[-] (e) at (2,-0.25) {$v_{n-1}$};
					\node[-] (y) at (1,0.1) {$y$};
					
					\node[-] (name) at (0.7,-3.7) {$v_i\in D^0$};
					
					\draw[-] (x) -- (a);
					\draw[dashed] (b) to[bend left=35] (a);
					\draw[-] (b) -- (c);
					\draw[-] (c) -- (d);
					\draw[dashed] (e) to[bend left=35] (d);
					\draw[-] (e) -- (y);
					
					\draw[-] (x) -- (c);
					\draw[-] (y) -- (c);
					
				\end{scope}
				
				\begin{scope}[xshift=2.5cm,yshift=-5cm]
					\node[-] (x) at (0,0.1) {$x$};
					\node[-] (a) at (-0.75,-0.25) {$v_2$};
					\node[color=red] (b) at (0,-3) {$v_{i}$};
					\node[color=red] (d) at (1,-3) {$v_{i+1}$};
					\node[-] (e) at (2,-0.25) {$v_{n-1}$};
					\node[-] (y) at (1,0.1) {$y$};
					
					\node[-] (name) at (0.7,-3.7) {$v_i,v_{i+1}\in D^1$};
					
					\draw[-] (x) -- (a);
					\draw[dashed] (b) to[bend left=45] (a);
					\draw[-] (b) -- (d);
					\draw[dashed] (e) to[bend left=45] (d);
					\draw[-] (e) -- (y);
					
					\draw[-] (x) -- (b);
					\draw[-] (y) -- (d);
					
				\end{scope}
				
				\begin{scope}[xshift=7cm,yshift=-5cm]
					\node[-] (x) at (0,0.1) {$x$};
					\node[-] (a) at (-0.75,-0.25) {$v_2$};
					\node[color=red] (b) at (-0.55,-2.8) {$v_{i}$};
					\node[-] (c) at (0.5,-3.2) {$v_{i+1}$};
					\node[color=red] (d) at (1.75,-2.8) {$v_{i+2}$};
					\node[-] (e) at (2,-0.25) {$v_{n-1}$};
					\node[-] (y) at (1,0.1) {$y$};
					
					\node[-] (name) at (0.7,-3.7) {$v_{i},v_{i+2}\in D^2$};
					
					\draw[-] (x) -- (a);
					\draw[dashed] (b) to[bend left=35] (a);
					\draw[-] (b) -- (c);
					\draw[-] (c) -- (d);
					\draw[dashed] (e) to[bend left=35] (d);
					\draw[-] (e) -- (y);
					
					\draw[-] (x) -- (b);
					\draw[-] (y) -- (d);
					
				\end{scope}
				
			\end{tikzpicture}
			\caption{Illustrations of sets defined in the proof of Lemma~\ref{l:general}}
		\end{figure}\label{f:sets_general_case}
		We need to define several sets. Some of them are illustrated in Figure~\ref{f:sets_general_case}:

		\begin{itemize}
			\item $S=\{v_i : v_ix,v_iy\notin E\}$
			\item $S^*=\{v_i : v_i\in S, v_{i-1}y,v_{i+1}x\in E\}$
			\item $S^0=\{v_i : v_i\in S, v_{i-1}x,v_{i+1}y\in E\}$
			\item $S^2=\{v_i : v_i\in S, v_{i+1}\in S\mbox{ or } v_{i-1}\in S\}$
		\end{itemize}
		
		\begin{itemize}
			\item $D^0=\{v_i : v_ix,v_iy\in E\}$
			\item $D^1=\{v_i,v_{i+1} : v_i,v_{i+1}\notin D^0, v_ix,v_{i+1}y\in E\}$ 
			\item $D^2=\{v_i,v_{i+2} : v_i,v_{i+2}\notin D^0, v_{i+1}\in S, v_ix,v_{i+2}y\in E\}$
			\item $D^x=\{v_i : v_i\in D^0\cup D^1\cup D^2 \mbox{ and } v_ix\in E\}$
			\item $D^y=\{v_i : v_i\in D^0\cup D^1\cup D^2 \mbox{ and } v_iy\in E\}$
		\end{itemize}

		Observe that $|D^2|=2|S^0|$ and $|D^x|=|D^y|=|D^0|+\frac{|D^1|}{2}+\frac{|D^2|}{2}$.

		A path $P(v_a,v_b)$ is a {\it segment} \blue{of $P$} if $v_a\in D^y$, $v_b\in D^x$ and for all $v_k$, $a<k<b$,  $v_k\notin (D^x\cup D^y)$. There are $|D^x|-1$ segments because for each $v_i \in D^x$, there is a $v_j \in D^y$ with $j \in \{i,i+1, i+2\}$. 

		\begin{claim}	\label{cf:raltion_S_et_D}
			$|S|=|D^0|+t-2$
		\end{claim}
		
		\begin{proofclaim} 
			
			Observe that $d(x)+d(y)=n-|S|+|D^0|-2$. Since $d(x)+d(y)=n-t$, the result follows. 
		\end{proofclaim}

		\begin{claim}\label{cf:Voisins_de_S*}
			$\forall v_i\in S^*$, $N(v_i)\subseteq \{v_j : j<i, v_{j+1}\in S\}\cup \{v_j : j>i,v_{j-1}\in S\}$. 
			
			In addition, $S^*$ is a stable set and $|N(S^*)|\leq 2|S| $. 
		\end{claim}

		\begin{proofclaim}
			
			Let $v_i\in S^*$. By definition, we have $v_{i-1}y, v_{i+1}x\in E$. 
			
			By (\ref{c:arretes_biais2}),  we have  $N(v_i)\subseteq \{v_j : j<i, v_{j+1}x\notin E\}\cup \{v_j : j>i,v_{j-1}y\notin E\}$. 
			
			By (\ref{c:triples}),  we have  $N(v_i)\subseteq \{v_j : j<i, v_{j+1}x,v_{j+1}y\notin E\}\cup \{v_j: j>i,v_{j-1}y ,v_{j-1}x\notin E\}$ and so $N(v_i)\subseteq \{v_j : j<i, v_{j+1}\in S\}\cup \{v_j : j>i,v_{j-1}\in S\}$.
			The next part of the claim directly follows. 
		\end{proofclaim}

		\begin{claim}	\label{cf:taille_S_S*}
			$|S^*|\geq |S|-2t+2$.
		\end{claim}
		
		\begin{proofclaim}
			There are $|D^x|-1$ segments. By (\ref{c:Sdanssegment}), we know that every segment contains at least one vertex from $S$. In addition, if a segment contains only one vertex from $S$, then this vertex is in $S^*$. There are at most $|S|-(|D^x|-1)$ segments containing more than 2 vertices from $S$. Hence, there are at least $|D^x|-1-(|S|-|D^x|+1)=2|D^x|-|S|-2$ segments containing only one vertex from $S$. Therefore $|S^*|\geq 2|D^x|-|S|-2$. Obviously,  we have  $|D^x|\geq |D^0|$. By (\ref{cf:raltion_S_et_D}),  we have  $|D^0| =  |S|-t+2$ and so $|S^*|\geq |S|-2t+2$.
		\end{proofclaim}

		\begin{claim}\label{cf:S*_vide}
			$S^*=\emptyset$ and $|S|\leq 2t-2$. 
		\end{claim}

		\begin{proofclaim}
			
			Suppose that $S^*\neq \emptyset$. By (\ref{cf:Voisins_de_S*}),  we have  $|N(v_i)|\leq |S|+1$,   for all $ v_i\in S^*$. Since $G$ is $\frac{3t-1}{2}$-tough,  we have  $3t-1\leq |N(v_i)|$ and so, $|S|\geq 3t-2$.

			Again, by (\ref{cf:Voisins_de_S*}), $S^*$ is a stable set and $|N(S^*)|\leq 2|S|$. In addition, $x\notin S^*\cup N(S^*)$. Hence, $G[V\sm N(S^*)]$ has at least $|S^*|+1$ connected components. Since $G$ is $\frac{3t-1}{2}$-tough :

			$ 
			\begin{array}{r c l}
				\frac{3t-1}{2}(|S^*|+1)  & \leq & 2|S| \\
				\Rightarrow (3t-1)(|S| -2t +3)   & \leq & 4|S| \mbox{ By (\ref{cf:taille_S_S*} ) }\\
				\Rightarrow (3t-1)|S| -(3t-1)(2t -3)   & \leq & 4|S|\\
				\Rightarrow (3t-5)|S|   & \leq & (3t-1)(2t-3)\\
				\Rightarrow (3t-5)|S|   & \leq & (3t-5)(2t)-(t-3)\\
				\Rightarrow |S|   & \leq & 2t -\frac{t-3}{3t-5} \mbox{ because } 3t-5 > 0\\
				\Rightarrow |S|   & \leq & 2t \mbox{ because } t \geq 3
			\end{array}
			$
			
			\vspace{2ex}
			
			By the previous observation,   we have  $|S|\geq 3t-2$ and so $3t-2\leq 2t$ and $t\leq 2$ a contradiction. The  second part of the claim follows from (\ref{cf:taille_S_S*}).
		\end{proofclaim}
		
		
		We now need to define three other sets : 
		
		\vspace{2ex}
		$ 
		\begin{array}{r c l}
			A  & = & D^0\cup D^1\cup S^0\cup (S^2\sm \{v_i:v_i\in S^2 \mbox{ and } v_{i+1}x\in E\})\\
		& & \cup \{v_i : v_i\notin S\cup \{x,y\} \mbox{ and }v_{i+1}\in S\}\\
			  &   &\cup \{v_i : v_i\notin S\cup \{x,y\}, v_ix\notin E  \mbox{ and } v_{i-1}\in S\} \\
			  &&\\
			X & = & (N(x)\sm A)\cup \{x\}\cup \{v_i\in S\sm S^0 : v_{i+1}x\in E\}\\
			&&\\
			Y & = & (N(y)\sm A)\cup \{y\}\cup \{v_i\in S\sm S^0 : v_{i-1}y,v_{i+1}y\in E\}
		\end{array}
		$
	
		\vspace{2ex}
		\begin{claim}	\label{cf:AXY_partition}
			$X$, $Y$ and $A$ form a partition of $V$.
		\end{claim}
		
		\begin{proofclaim}
			
			It is easy to check that $X\cap A=\emptyset$ and $Y\cap A=\emptyset$. 
			
			Suppose that there is $v_i\in X\cap Y$. By the definition of $X$ and $Y$, either $v_i\in S$ or $v_i\in (N(y)\cap N(y)\sm A)$. Since $D^0\cup S^0\subseteq A$,  we have   $v_i\in S\sm S^0$. By definition of $X$ and $Y$, it follows that $v_{i-1}y,v_{i+1}x\in E$ and so $v_i\in S^*$, a contradiction to (\ref{cf:S*_vide}). Hence $X\cap Y=\emptyset$. 
			
			
			Now, we prove that $V\subseteq X\cup Y \cup A$. Obviously,   we have  $N(x)\cup N(y)\cup \{x,y\}\subseteq X\cup Y \cup A$. Suppose there is $v_i\in S$ such that $v_i\notin X\cup Y\cup A$. By (\ref{cf:S*_vide}) and since $S^0\in A$,   we have  $v_i\notin S^0\cup S^*$. Hence either $v_i\in S^2$ or $v_{i-1}x,v_{i+1}x\in E$ or $v_{i-1}y,v_{i+1}y\in E$. By definition of $X$,   we have  $v_{i+1}x\notin E$. But now, if $v_i\in S^2$ then $v_i\in A$ and if $v_{i-1}y,v_{i+1}y\in E$ then $v_i\in Y$. Both cases contradicts $v_i\notin X\cup Y\cup A$.	
		\end{proofclaim}

		\begin{claim}\label{cf:deconectons}
			$G[V\sm A]$ has at least two connected components. 
		\end{claim}
		
		\begin{proofclaim}
			We will prove that there is no edge between $X$ and $Y$.
			
			Suppose by contradiction that $v_i\in X$, $v_j\in Y$ such that $v_iv_j\in E$. \blue{By definitions of $X$ and $Y$, we have $v_i\neq x$, $v_j\neq y$. Consider the vertex $v_i$. We claim that  $v_i y \not \in E$. If $v_i \in S$ then the claim holds by definition of $S$. If $v_i \not\in S$, then $v_i \in N(x) \sm D^0$, and thus $v_i y \not \in E$.  A similar argument shows that $v_j x \not \in E$. }
			
			
			
			\blue{We will prove that $v_{i+1}x\in E$. If $v_i \in (S \sm S^0)$, then by definition of $X$, $v_{i+1}x\in E$ and we are done. So we may assume that $v_i x \in E$, and therefore $v_i \not\in S$. If $v_{i+1} \in S$, then by definition $v_i \in A$, a contradiction. So, we may assume that $v_{i+1} y \in E$. If  $v_{i+1} x \not\in E$, then $v_i \in D^1$, a contradiction. We have proved that $v_{i+1}x\in E$.}
				
			\blue{Suppose that $j < i$. By 	(\ref{c:triples}),   we have  $v_{j+1} y \not \in E$. By definition of $Y$, $v_j \in N(y)$. We must have $v_{j+1} \not\in S$, for otherwise $v_j \in A$, a contradiction. This implies $v_{j+1} x \in E$. But this is a contradiction to (\ref{c:arretes_qui_croise}). So,  $i < j$.   } 
			
			\blue{Suppose that $v_{j+1}y\notin E$. By (\ref{c:arretes_qui_croise}),  we have   $v_{j-1} y \not \in E$. The definition of $Y$ implies that $v_j \in N(y)$, and therefore $v_j \not \in S$. If $v_{j+1} \in S$, then  $v_j \in A$, a contradiction. It follows that  $v_{j+1} x \in E$. But this is a contradiction to (\ref{c:arretes_qui_croise}). So $v_{j+1}y\in E$.}
		
			
			By (\ref{c:arretes_biais2}),  we have   $v_{j-1}y \notin E$. By definition of $Y$,  we have   $v_jy\in E$. If $v_{j-1}x\in E$ then  $v_j\in D^1$, a contradiction. It follows that $v_{j-1}\in S$. By definition,  we have  $v_j \not \in D^0$. So $v_j x \notin E$. This implies $v_j \in A$, a contradiction.  
		\end{proofclaim}
		
		\begin{claim}	\label{cf:ceux_qui_restent}
			$|\{v_i : v_i\in S^2\mbox{ and } v_{i+1}x\in E\}|\geq |D^x|-1$.
		\end{claim}
		
		\begin{proofclaim}		
			Let $P(v_a,v_b)$ be a segment. By (\ref{c:Sdanssegment}) and (\ref{cf:S*_vide}),   we have $|\{v_{a+1},\dots , v_{b-1}\}|\geq 2$. Let $\alpha$ be the smallest integer such that $a<\alpha<b$, $v_\alpha \in S$, $v_{\alpha +1}x\in E$. The integer $\alpha$ exists by (\ref{c:arretes_biais2}). By (\ref{cf:S*_vide}), we have $v_{\alpha -1}y\notin E$.  
			
			Suppose that  $v_{\alpha -1}x\in E$.  By (\ref{c:Sdanssegment}), (applied on $P(v_a,v_{\alpha-1})$) there exists $\beta$ with $a<\beta<\alpha-1$ such that $v_\beta \in S$ and $v_{\beta +1}x\in E$.  But this is a contradiction to the choice of $\alpha$. 
			
			Therefore, $v_{\alpha -1}\in S$ and $v_\alpha \in \{v_i : v_i\in S^2\mbox{ and } v_{i+1}x\in E\}$. Hence every segment contains a vertex from $\{v_i : v_i\in S^2\mbox{ and } v_{i+1}x\in E\}$.  Since there are $|D^x|-1$ segments, the result holds. 
		\end{proofclaim}
		
		\begin{claim}\label{cf:taille1A}
			$|A|\leq |S|+t$.
		\end{claim}
		
		\begin{proofclaim}
			By (\ref{cf:ceux_qui_restent}),  we have $|\{v_i : v_i\in S^2\mbox{ and } v_{i+1}x\in E\}|\geq |D^x|-1$. Consider the following relations.
			
			\vspace{2ex}
			\noindent$ \begin{array}{r c l}
				|\{v_i : v_i\notin S\mbox{ and } v_{i+1}\in S\}| & = & |\{v_i : v_i\notin S\mbox{ and } v_{i+1}\in S^2\}|\\
				& &    +  |\{v_i : v_i\notin S\mbox{ and } v_{i+1}\in S\sm S^2\}| \\
				&\leq& \frac{|S^2|}{2}+|S\sm S^2|\\
				& \leq & \frac{|S^2|}{2}+|S|-|S^2|\\
				& \leq & |S|-\frac{|S^2|}{2}\\
			\end{array}$
			
			\vspace{2ex}
			\noindent$\begin{array}{r c l}
				|\{v_i : v_i\notin S,  v_{i-1}\in S^2\mbox{ and } v_ix\notin E\}| & \leq & |\{v_i : v_i\notin S \mbox{ and } v_{i-1}\in S^2\}\sm \\
				& & \{v_i : v_i\notin S, v_{i-1}\in S^2\mbox{ and } v_ix\in E\}| \\
				& \leq & \frac{|S^2|}{2}-| \{v_i : v_i\notin S, v_{i-1}\in S^2\mbox{ and } v_ix\in E\}|\\
				& \leq & \frac{|S^2|}{2}-|\{v_i : v_i\in S^2\mbox{ and } v_{i+1}x\in E\}|\\
				& \leq & \frac{|S^2|}{2} - |D^x|+1
			\end{array}$
			
			\vspace{2ex}
			\noindent$\begin{array}{r c l}
				|\{v_i : v_i\notin S, v_{i-1}\in S\mbox{ and } v_ix\notin E\}| & \leq & |\{v_i : v_i\notin S, v_{i-1}\in S^2\mbox{ and } v_ix\notin E\}|+  |S\sm S^2| \\
				& \leq & \frac{|S^2|}{2} - |D^x|+1+|S|-|S^2|\\
				& \leq & |S|-\frac{|S^2|}{2}-|D^x|+1
			\end{array}$
			
			\vspace{2ex}
			By combining all this, we get: 
			
			\vspace{2ex} 
			$ 
			\begin{array}{r c l} 
				
				|A|& \leq& |D^0|+|D^1|+|S^0|+|S^2|-|D^x|+1+|S|-\frac{|S^2|}{2}+|S|-\frac{|S^2|}{2}-|D^x|+1 \\
				& \leq & |D^0|+|D^1|+|S^0|-2|D^x|+2|S|+2\\
				& \leq & |D^0|+|D^1|+\frac{|D^2|}{2}-2|D^x|+2|S|+2 \mbox{ by remarks about } S^0\\
				& \leq & \blue{\frac{|D^1|}{2}-|D^x|+2|S|+2} \mbox{ by remark about } D^x\\
				& \leq &  2|S|-|D^0| +2 \mbox{ by remarks about } D^x \mbox{ and } \frac{|D^2|}{2}\geq 0 \\
				& \leq &  2|S|-|S|+t-2+2\mbox{ by (\ref{cf:raltion_S_et_D})}\\
				& \leq &  |S|+t
			\end{array}
			$
			
			\vspace{2ex}
		\end{proofclaim}
		
		We continue the proof of the lemma. Since $G$ is $\frac{3t-1}{2}$-tough, we have

		$
		\begin{array}{r c l}
			
			\frac{3t-1}{2}c(G[V\sm A]) & \leq& |A| \\
			\Rightarrow 3t-1	& \leq & |A| \mbox{ by (\ref{cf:deconectons}) }\\
			\Rightarrow 3t-1 & \leq &  |S|+t \mbox{ by (\ref{cf:taille1A})}\\
			\Rightarrow 2t-1 & \leq &  |S|
		\end{array}
		$
		
		But, $|S|\leq 2t-2$ by (\ref{cf:S*_vide}), a contradiction. 	
	\end{proof}

		
	

	
%
	
	\begin{corollary}
		Let $G$ be a $t'$-tough graph ($t'\geq \frac{5}{2}$) with $x,y\in V(G)$ such that $d(x)+d(y)\geq n-\frac{2t'+1}{3}$.
		
		$G$ is Hamiltonian if and only if $G^*=(V(G),E(G)\cup \{xy\})$ is Hamiltonian. $\qed$
	\end{corollary}

	Now, the $t$-Closure Lemma follows directly:
	
	\begin{lemma}[$t$-Closure Lemma]\label{l:closure+}
		
		For $t\geq 2$, a $\frac{3t-1}{2}$-tough graph $G$ is Hamiltonian if and only if its $t$-closure $G^{*t}$ is. $\qed$
	\end{lemma}

	\section{Hamiltonian cycle and property $P(t)$}\label{s:Hamiltonian_and_tough}
	
	We now use Lemma~\ref{l:closure+} to prove  Theorem~\ref{t:notre}. 	
	
	Recall the following definition:  A \textit{universal clique} in a graph is set of vertices in the graph that are adjacent to every other vertex of the graph. 
	 
	 Consider a $t$-tough graph $G$ with degree sequence $d_1 \leq d_2 \leq \dots \leq d_n$.
	
	Let $P(t)$ be the predicate : 
	
	$$\forall i <\frac{n}{2},\mbox{ if } d_i\leq i \mbox{ then } d_{n-i+t}\geq n-i$$
	
	 We restate Theorem \ref{t:notre} below.
	
	\noindent {\bf Theoreom \ref{t:notre}} {\it 	If $G$ is a $4$-tough graph satisfying $P(4)$, then $G$ is Hamiltonian.}

	\begin{proof}
		Let $G$ be a $4$-tough graph satisfying $P(4)$.  
		
		Let $d_1,d_2,\dots , d_n$ be the degree sequence of $G$ and let $x_i$ to be the vertex of degree $d_i$. 
		By applying Lemma~\ref{l:closure+} with $t=3$, we may  assume that
		
		\vspace{2ex}
		\begin{claim}\label{c:edge}
		For all $i$ and $j$ such that $d_i+d_j\geq n-3$, we have $x_ix_j\in E$. 
		\end{claim}
		
	 	If, $d_i > i$ for all $i<\frac{n}{2}$,  then  $G$ satisfies the condition of Theorem~\ref{t:chvatal} and so, $G$ is Hamiltonian. Hence, we may assume that there exists $k<\frac{n}{2}$ such that $d_k\leq k$. Choose $k$ to be minimum and so for all $i< k$,  $d_i>i$. 
	 	
	 	By the definition of degree sequece, we have $d_{k-1}\leq d_k\leq k$. Since $d_{k-1}>k-1$ (because $k-1< k$), it follows that $d_k=k$ and $d_{k-1} = k$. 
	 	
	 	Since $G$ is $4$-tough, $4c(G \setminus N(x_k))\leq |N(x_k)|$. Hence $8\leq d_k=k$ and so $k\geq 8$. 
		
		For all $\alpha<\frac{n}{2}$, define $U^\alpha=\{x_j : d_j\geq n-\alpha\}$.
		
		\vspace{2ex}
		\begin{claim}\label{c:presque_clique_univ}
			For all $\alpha<\frac{n}{2}$, $U^\alpha$ is a clique complete to  $\{x_i : d_i\geq \alpha-3\}$.
		\end{claim}

		\begin{proofclaim}
			
		For any $x_u\in U^\alpha$ and any $x_i\in \{x_i : d_i\geq \alpha-3\}$, we have $d_u+d_i\geq n-\alpha+\alpha-3$.  By (\ref{c:edge}), we have $x_ux_i\in E$. Hence  $U^\alpha$ is complete  to  $\{x_i : d_i\geq \alpha-3\}$. Since $U^\alpha \subseteq \{x_i : d_i\geq \alpha-3\}$, it follows that $U^\alpha$ is a clique.
		\end{proofclaim}
		
		\begin{claim}	\label{cp4:cas_universal}
			Let $\alpha<\frac{n}{2}$. If for all $i$ with $d_i<\alpha-3$, we have  $d_i\geq i-2$, then $U^\alpha$ is a universal clique.
		\end{claim}
		
		\begin{proofclaim}
			Suppose that $\alpha$ satisfies the hypothesis of (\ref{cp4:cas_universal}). 
			
			Suppose by contradiction that $U^\alpha$ is not a universal clique. That is, there exists $x_u\in U^\alpha$ and $x_i\notin N(x_u)$. Choose $i$ to be maximum (i.e., for all $j>i$, $u \not= j$, we have  $x_jx_u\in E$). By (\ref{c:presque_clique_univ}), we have $d_i<\alpha-3$. By the hypothesis of the claim, we have $d_i\geq i-2$. It follows that $d_i = i-2$. By the maximality of $i$, we know that $x_u$ is adjacent to all vertices $x_j$ with $j > i, j \not= u$, that is, 
			$d_u\geq n-i-1$. Hence $d_u+d_i\geq n-i-1+i-2=n-3$, a contradiction to (\ref{c:edge}).
		\end{proofclaim}

		In the following, denote by $\Omega$ the maximum universal clique in $G$ (with the maximum number of vertices). Note that $\Omega \leq \frac{n}{2} - \frac{3}{2}$, for otherwise  by (\ref{c:edge}), $G$ is a clique, a contradiction. 

		\vspace{2ex}
		\begin{claim}	\label{cp4:tailleOmega}
			$|\Omega|\leq k-2$ 
		\end{claim} 
		
		\begin{proofclaim}
			Suppose by contradiction that 	$|\Omega|\geq k-1$. Since $\Omega$ is a universal clique, for every vertex $x_i\in V(G)$, we have $\Omega\subseteq N(x_i) \cup\{x_i\}$, and so $d_i\geq |\Omega| \geq k-1$. If $|\Omega| > k$, then $d_1 > k$, a contradiction. Hence $ k-1 \leq |\Omega|\leq d_k = k$.

			Set $\kappa = \{x_i:i\leq k\}$. Note that $|\kappa|=k$. By definition, for all $x_i\in \kappa$, we have $d_i\leq d_k=k$ and $|N(x_i)\sm \Omega|\leq d_i-|\Omega|\leq 1$.  
			
			Define $D_0=\{x_i\in \kappa : |N_{\kappa}(x_i)|=0 \} $ and $D_1=\{x_i\in \kappa : |N_{\kappa}(x_i)|=1 \} $. It follows from this definition that $|D_0| + |D_1| = k$ and $c(G[\kappa]) \geq |D_0| + \frac{1}{2} |D_1|$.
			
			Suppose that $| \Omega | = k$. Since every vertex $v_i \in \kappa$ is adjacent to all vertices of $\Omega$ and $d_i \leq |\Omega|$, $v_i$ has no neighbor in $V \sm \Omega$. It follows that $N(\kappa) = \Omega$ and $\kappa$ is a stable set ($D_1 = \emptyset$). Consider the cutset $\Omega$. We have $c(G \sm \Omega) \geq |\kappa| = k$ and $|\Omega| = k$, a contradiction to the assumption that $G$ is 4-tough.
			
			 So, we may assume $| \Omega | = k -1$. Each vertex in $\kappa$ is adjacent to at most one vertex in $G \sm \Omega$. It follows that $|N(\kappa)| = |\Omega| + |N(\kappa) \sm \Omega| \leq (k-1) + k = 2k - 1$. We also have $c(G[\kappa]) \geq |D_0| + \frac{1}{2} |D_1| \geq \frac{1}{2} (|D_0| + |D_1|) \geq \frac{k}{2}$. Consider the cutset $N(\kappa)$. Note that $c(G \sm N(\kappa)) \geq c(G[\kappa]) \geq \frac{k}{2}$. Since $G$ is 4-tough, we have $2k-1 \geq |N(\kappa)|\geq 4c(G[\kappa])\geq 2k$,
			 a contradiction.
		\end{proofclaim}
		
		\begin{claim}	\label{cp4:taille_U}
			For all $\alpha<\frac{1}{2}$ such that $d_\alpha \leq \alpha$, we have $|U^\alpha|\geq \alpha-3$. 
		\end{claim}
		
		\begin{proofclaim}
			Consider a vertex $v_{\alpha}$ with $d_{\alpha} \leq \alpha < \frac{1}{2}$. 
			Since $G$ satisfies $P(4)$, we have $d_{n- \alpha + 4} \geq n-\alpha$. This implies there are at least $ n-(n-\alpha+4)+1 = \alpha - 3$ vertices of degree at least $n - \alpha$. 
		\end{proofclaim}
		
		\begin{claim}	\label{cp4:lesAutres}
			For all $\alpha$ such that $ k+2 \leq \alpha < \frac{n}{2}$, we have  $d_\alpha > \alpha$.
		\end{claim}
		
		\begin{proofclaim}
			Suppose there exists $\alpha$ such that $ k+2 \leq \alpha < \frac{n}{2}$ and  $d_\alpha \leq \alpha$. Choose such $\alpha$ to be minimum.

			We will prove that $d_\alpha \geq \alpha -2$.  If $\alpha =k+2$ then $d_\alpha\geq d_k=k=\alpha -2$. If $\alpha >k+2$ then $\alpha -1\geq k+2$. By the choice of $\alpha$, it follows that $\alpha -1 < d_{\alpha-1}\leq d_\alpha \leq \alpha$ and so $d_\alpha =\alpha$.  In both cases, we have $d_\alpha \geq \alpha -2$.

			Consider a vertex $v_i$ with $k+2 \leq i < \alpha$. By the minimality of $\alpha$, we have $d_i > i$. The choice of $k$ implies that $d_i > i$ for $1 \leq i < k$, and that $d_k = k$. All of this implies $d_i \geq i$ for $1 \leq i \leq \alpha$. In addition, for all $i > \alpha$, we have $d_i \geq d_\alpha \geq \alpha -2$.  By (\ref{cp4:cas_universal}), $U^\alpha$ is a universal clique. By (\ref{cp4:taille_U}), we have $|U^\alpha|\geq \alpha -3$. By (\ref{cp4:tailleOmega}), we have $k-2\geq |\Omega|\geq |U^\alpha|\geq \alpha -3$ and so $k+1\geq \alpha$, a contradiction. 
		\end{proofclaim}
		
		\begin{claim}\label{cp4:onFixeTout}
			$n$ is even, $d_{k+1}=k $ and $k=\frac{n}{2}-2$.
		\end{claim}
		
		\begin{proofclaim}
			We will first prove that $d_{\lfloor\frac{n-1}{2}\rfloor}=\frac{n}{2}-2$. For simplicity, let $\nu =\lfloor\frac{n-1}{2}\rfloor$.
			
			Observe that, since $d_k=k$, we have $d_{k+1}\geq k$. By (\ref{cp4:lesAutres}) and the choice of $k$, for all $i<\frac{n}{2}$, we have  $d_i> i$ or $i\in \{k,k+1\}$. Hence $d_i\geq i-1$ for $i < \frac{n}{2}$. Therefore $d_{\nu}\geq \nu -1\geq \frac{n}{2}-2$ (because $\nu\geq \frac{n}{2}-1$). 
			
			Suppose that  $d_{\nu}\geq \frac{n-1}{2}-1$.

			If $d_{\nu}=n-1$ then  $|\Omega|> \frac{n}{2}$. It follows from (\ref{cp4:tailleOmega}) that $k - 2 \geq \frac{n}{2}$, a contradiction to the choice of $k$ ($k < \frac{n}{2}$). Hence there is some $i$ such that $x_ix_{\nu }\notin E$. Choose $i$ to be maximum. By (\ref{c:edge}), we have $d_i<n-3-d_\nu \leq n-3-( \frac{n-1}{2}-1)= \frac{n-1}{2}-1 \leq d_{\nu}$. Hence $i< \nu$. By previous observation,  we have $d_i\geq i-1$. By the choice of $i$, $x_{\nu}$ is adjacent to all vertices $x_j$ with $j > i, j \not= \nu$.  So, we have $d_\nu \geq n-i-1$ and $d_{\nu}+d_i\geq n-2$, a contradiction to (\ref{c:edge}). We conclude that $\frac{n}{2} -2  \leq d_{\nu} < \frac{n-1}{2} -1$.  
		
		 Now, since $d_\nu$ is an integer, we know that $n$ is even, $d_\nu= \frac{n}{2}-2$ and $\nu=\frac{n}{2}-1$. Hence $d_{\frac{n}{2}-1}= \frac{n}{2}-2$. Therefore  $d_{\frac{n}{2}-2}\leq \frac{n}{2}-2$. If $k < \frac{n}{2} - 2$, then
		 (\ref{cp4:lesAutres}) implies $d_{\frac{n}{2}-1} > \frac{n}{2} -1$, a contradiction. The choice of $k$ implies $k = \frac{n}{2}-2$, and  so $d_{k+1}=k $.
		\end{proofclaim}

		\begin{claim}\label{cp4:FxeOmega}
			$\Omega=U^{\frac{n}{2}-1}$ and $|\Omega|=\frac{n}{2}-4$.
		\end{claim}
		
		\begin{proofclaim}
			By (\ref{cp4:onFixeTout}), we have  $k=\frac{n}{2}-2$ and  $d_{k+1}=k<\frac{n}{2}-1 $. By (\ref{cp4:cas_universal}), $U^{\frac{n}{2}-1}$ is a universal clique. Hence, by (\ref{cp4:taille_U}), we have  $|\Omega|\geq |U^{\frac{n}{2}-1}|\geq \frac{n}{2}-4$ . By (\ref{cp4:tailleOmega}), we have  $k-2\geq |\Omega|\geq \frac{n}{2}-4$ and so $|\Omega|= \frac{n}{2}-4$.
		\end{proofclaim}

		\begin{claim}\label{cp4:Fxeautre}
			For all $i$, either $x_i\in \Omega$ or $d_i\leq \frac{n}{2}-2$.
		\end{claim}
		
		\begin{proofclaim}
			
			Suppose by contradiction that there exists $\alpha$ such that $x_\alpha\notin \Omega$ and  $ d_\alpha \geq \frac{n}{2}-1$. Choose $\alpha$ to be minimum (i.e., for all $j<\alpha$, $d_j\leq \frac{n}{2}-2$).
			
			
			Since $x_\alpha \notin U^{\frac{n}{2}-1} \subseteq \Omega$, we have $d_\alpha< n-\frac{n}{2}+1$. Hence  $d_\alpha\leq \frac{n}{2}$.  
		
			Let $j$ be the largest subscript such that $x_jx_\alpha\notin E$. We must have $d_j < \frac{n}{2} - 2$, for otherwise we have $d_j + d_{\alpha} \geq \frac{n}{2} - 1 + \frac{n}{2} - 2 = n - 3$, a contradiction to (\ref{c:edge}). It follows that $j < k$ (recall that $k = \frac{n}{2} -2$) and so by choice of $k$, we have $d_j > j$. Now, the choice of $j$ implies that $x_{\alpha}$ is adjacent to all vertices $x_i$ with $i > j, i \not= \alpha$. That is, $d_{\alpha} \geq n - j -1$. But now we have $d_j + d_{\alpha} > j + n - j - 1 = n -1$, a contradiction to (\ref{c:edge}). 
			
		\end{proofclaim}

		Since $G$ is $4$-tough, we have $c(G[V\sm \Omega])\leq \frac{1}{4} |\Omega|$. Hence $c(G[V\sm \Omega])\leq \frac{n}{8}-1$, by (\ref{cp4:FxeOmega}).
		
		Let $C_1,C_2,\dots C_l$ be the connected components of $G[V\sm \Omega]$ (with $l\leq \frac{n}{8}-1$).
		
		For all $x_i\in V(G)\sm \Omega$, we have   $|N_{V(G)\sm \Omega}(x_i)| \leq d_i - |\Omega|$. By (\ref{cp4:FxeOmega}) and (\ref{cp4:Fxeautre}), we have $|N_{V(G)\sm \Omega}(x_i)|\leq 2$. Hence, for all $i\leq l$, $C_i$ is either a path, a cycle or an isolated vertex. Hence there exits a path $P^i$  (possibly not induced) covering $C_i$. Denote $a_i$ and $b_i$ the endpoints of $P^i$ (possibly $a_i=b_i$). 
		
		Since $|\Omega|\geq \frac{n}{2}-4$, for all $i\leq \frac{n}{8}-1$, we have  $x_{n-i}\in \Omega$. 
		
		Now build the following cycle (see Figure~\ref{f:hamilpaht} with 3 connected components in $G[V\sm \Omega]$). Start with $x_n$. For all $i\leq l$, add the path $a_iP^ib_ix_{n-i}$. Finally, add every vertex remaining in $\Omega$ to the cycle. This is a Hamiltonian cycle  of $G$. 
	\end{proof}
	
	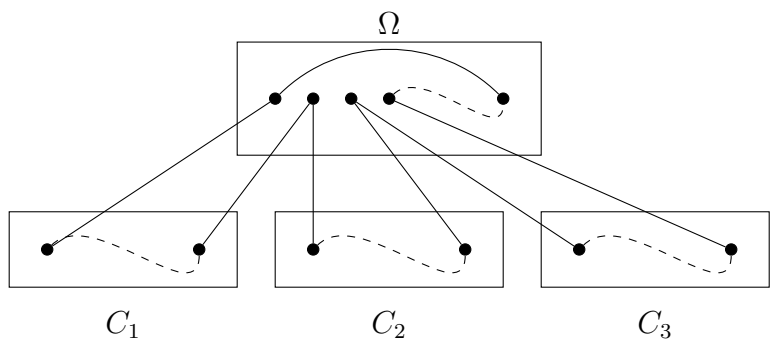
\begin{figure}
		\centering
		\begin{tikzpicture}[scale=1]
			
			\begin{scope}[xshift=1cm,yshift=0cm]
				
				\draw	(0,0.75) rectangle (4,2.25);
				
				\node[vertex] (ao) at (0.5,1.5) {};
				\node[vertex] (bo) at (1,1.5) {};
				\node[vertex] (co) at (1.5,1.5) {};
				\node[vertex] (do) at (2,1.5) {};
				\node[vertex] (eo) at (3.5,1.5) {};
				
				\node[-] (name) at (2,2.5) {$\Omega$};
				
				\draw[dashed] (do) to[out=45,in=-90] (eo);
				
			\end{scope}
			
			\begin{scope}[xshift=-2cm,yshift=-1cm]
				
				\draw	(0,0) rectangle (3,1);
				
				\node[vertex] (a1) at (0.5,0.5) {};
				\node[vertex] (b1) at (2.5,0.5) {};
				
				\node[-] (name) at (1.5,-0.5) {$C_1$};
				
				\draw[dashed] (a1) to[out=45,in=-90] (b1);
				
			\end{scope}
			
			\begin{scope}[xshift=1.5cm,yshift=-1cm]
				
				\draw	(0,0) rectangle (3,1);
				
				\node[vertex] (a2) at (0.5,0.5) {};
				\node[vertex] (b2) at (2.5,0.5) {};
				
				\node[-] (name) at (1.5,-0.5) {$C_2$};
				
				\draw[dashed] (a2) to[out=45,in=-90] (b2);
				
			\end{scope}
			
			\begin{scope}[xshift=5cm,yshift=-1cm]
				
				\draw	(0,0) rectangle (3,1);
				
				\node[vertex] (a3) at (0.5,0.5) {};
				\node[vertex] (b3) at (2.5,0.5) {};
				
				\node[-] (name) at (1.5,-0.5) {$C_3$};
				
				\draw[dashed] (a3) to[out=45,in=-90] (b3);
				
			\end{scope}
			
			\draw[-] (ao) -- (a1);
			\draw[-] (b1) -- (bo);
			\draw[-] (a2) -- (bo);
			\draw[-] (b2) -- (co);
			\draw[-] (co) -- (a3);
			\draw[-] (do) -- (b3);
			\draw[-] (eo) to[bend right=45] (ao);

		\end{tikzpicture}
		\caption{The final Hamiltonian cycle for proof of Lemma~\ref{t:notre}}
	\end{figure}\label{f:hamilpaht}

	\section{Conclusion and open problems}
	
	In this paper, we study properties of tough graphs, and an extension of  Chv\'atal's degree condition for the existence of Hamiltonian cycles in a graph. Chv\'atal's theorem (Theorem~\ref{t:chvatal}) is in a sense best possible because (as we mentioned in the Introduction) if a graph $G$ does not satisfy Chv\'atal's degree condition, then there is a graph $G'$ that degree-majorizes $G$ and is not Hamiltonian. A necessary condition for a graph to be Hamiltonian is that it has to be 1-tough. So, one may ask the following: 
	
	\noindent {\bf Open Problem} {\it 
	Characterize a degree condition $P$ such that if a 1-tough graph satisfies $P$ then it is Hamiltonian, but if $G$ does not satisfy $P$ then there is a 1-tough graph $G'$ such that $G'$ degree-majorizes $G$, but $G'$ is not Hamiltonian. 
    }
	
	Theorem~\ref{t:Chinh} is a first step in this direction. Jung \cite{jung_1987} proved that if $G$ is a 1-tough $n$-vertex  graph with $n \geq 11$ and $d(x) + d(y) \geq n-4$ for any two non-adjacent vertices $x$ and $y$, then $G$ is Hamiltonian. Thus the degree condition in the Open Problem must contain the sequence $n-2, \ldots, n-2$. 
	 
	Even though the $t$-Closure Lemma is a step forward to prove Conjecture~\ref{t:conj_generale}, it is not enough to prove the step further that is when $t\geq 5$. Other tools would be needed. 
	Nevertheless, we believe that the $t$-Closure Lemma could be useful in investigating the relationship between tough graphs and Hamiltonian cycles regardless of the degree condition. 
	
	\vspace*{1em}
	
	\noindent {\bf Acknowledgement} This work was supported by the Canadian Tri-Council Research
	Support Fund. The author C.T. Ho\`ang was supported by individual NSERC Discovery Grant.
	
	\bibliography{tough-hamiltonian-final}
\end{document}